\documentclass{amsart}
\usepackage{amssymb}
\usepackage{amsthm}
\usepackage{indentfirst}
\usepackage{amsmath}
\usepackage{amscd}
\usepackage{xypic}

\newtheorem{Theorem}{Theorem}[section]
\newtheorem{Definition}[Theorem]{Definition}
\newtheorem{Proposition}[Theorem]{Proposition}
\newtheorem{Lemma}[Theorem]{Lemma}

\newtheorem{Assumption-Notation}[Theorem]{Assumption-Notation}

\newtheorem{Remark}[Theorem]{Remark}
\newtheorem{Corollary}[Theorem]{Corollary}

\numberwithin{equation}{section}

\begin{document}

\title{A note on the linear systems on the projective bundles over Abelian varieties}

\author{Lei Zhang}

\address{College of Mathematics and information Sciences\\Shaanxi Normal University\\Xi'an 710062\\P.R.China}

\email{lzhpkutju@gmail.com}

\thanks{I would like to thank Dr. Hao Sun who discussed with me and helped me to improve Theorem \ref{key}. I also express my gratitude to two anonymous refrees who gave a lot of useful suggestions.}


\date{March 18, 2012 and, in revised form, August 15, 2012.}


\begin{abstract}
 It is well known that for an ample line bundle $L$ on an Abelian variety $A$, the linear system $|2L|$ is base point free, and $3L$ is very ample, moreover the map defined by the linear system $|2L|$ is well understood (cf. Theorem \ref{oldth}). In this paper we generalized this classical result and give a new proof using the theory developed by Pareschi and Popa in \cite{PP} (cf. Theorem \ref{key}).
\end{abstract}

\maketitle

\section{Introductions}\label{intro}

\textbf{Conventions:} All the varieties are assumed over $\mathbb{C}$. For a variety $X$ and a vector bundle $E$ on it, $\mathbb{P}_X(E)$ is defined as $Proj_{\mathcal{O}_X}(\oplus_kS^k(E^*))$, and $\mathcal{O}_{\mathbb{P}_X(E)}(1)$ denotes the anti-tautological bundle. Let $f: X \rightarrow Y$ be a morphism between two
smooth projective varieties, denote by $D^b(X)$ and $D^b(Y)$ their
bounded derived categories of coherent sheaves, and by $Rf_*$ and $Lf^*$
the derived functors of $f_*$ and $f^*$ respectively. For $E \in D^b(X)$,
$E^*$ denotes its dual $R\mathcal{H}om(E, \mathcal{O}_X)$, and we say that
$E$ is a sheaf if it is quasi-isomorphic to a sheaf in $D^b(X)$. If $X = X_1 \times X_2 \times ... \times X_r$ denotes the product of $r$ varieties, then $p_i$ denotes the projection from $X$ to the $i$-th factor $X_i$.
If $A$ denotes an Abelian variety, then $\hat{A}$ denotes its dual $Pic^0(A)$, $\mathcal{P}$ denotes the Poincar$\acute{e}$ line bundle on
$A \times \hat{A}$, and the Fourier-Mukai transform $R\Phi_{\mathcal{P}}: D^b(A) \rightarrow D^b(\hat{A})$ w.r.t. $\mathcal{P}$ is defined as
$$R\Phi_{\mathcal{P}}(\mathcal{F}) := R(p_2)_*(Lp_1^*\mathcal{F} \otimes \mathcal{P});$$
similarly $R\Psi_{\mathcal{P}}: D^b(\hat{A}) \rightarrow D^b(A)$ is defined as
$$R\Psi_{\mathcal{P}}(\mathcal{F}) := R(p_1)_*(Lp_2^*\mathcal{F} \otimes \mathcal{P})$$
For a sheaf $\mathcal{F}$ on an Abelian variety $A$, if $R\Phi_{\mathcal{P}}(\mathcal{F}) \cong R^0\Phi_{\mathcal{P}}(\mathcal{F})$, we say $\mathcal{F}$ satisfies $IT^0$. If $a: X \rightarrow A$ denotes a map to an Abelian variety, then $\mathcal{P}_a : = (a \times id_{\hat{A}}) ^*\mathcal{P}$, and for $\mathcal{F} \in D^b(X)$, $R\Phi_{\mathcal{P}_a}(\mathcal{F})$ is defined similarly; and if $\alpha \in \hat{A}$, we often denote the line bundle $a^*\alpha \in Pic^0(X)$ by $\alpha$ for simplicity.

For the linear systems over an Abelian variety, the following classical results are well known.
\begin{Theorem}\label{oldth}
Let $A$ be an Abelian variety, and $L$ an ample line bundle on it. Then
$3L$ is very ample, $|2L|$ is base point free; and if
$2L$ is not very
ample, then
\begin{itemize}
\item[(i)]{$A = A_1 \times A_2$ and $L \cong L_1 \boxtimes L_2$
where $L_i$ is a line bundle on $A_i$ and at least one of $(A_i,L_i)$ is a
principally polarization (\cite{Ram}, \cite{Oh});}
\item[(ii)]{moreover if $A$ is simple and $(A,L)$ is a principal polarization,
then $L$ is symmetric up to translation, and the
map $\phi$ defined by $|2L|$ coincides with the quotient map $A
\rightarrow A/(-1)_A$ (\cite{LB} Sec. 4.5, 4.8).}
\end{itemize}
\end{Theorem}

In this paper, as a generalization, we prove
\begin{Theorem}\label{key}
Let $A$ be an Abelian variety, $E$ an $IT^0$ vector bundle on it,
$P = \mathbb{P}_A(E^*)$ with anti-tautological line bundle
$\mathcal{O}_P(1)$. Then
$\mathcal{O}_P(3)$ is very ample, and
the linear system $|\mathcal{O}_P(2)|$ is base point free and hence defines a morphism $\phi$. Moreover $\phi$ is not birational if and only if $A \cong A_1 \times A_2$, and $E \cong L_1 \boxtimes E_2$ where $L_1$ is a line bundle on $A_1$ and $E_2$ is a vector bundle on
$A_2$ such that either $\chi(A_1, L_1)
= 1$ or $\chi(A_2, E_2) = 1$.

In particular if $A$ is simple, then $\phi$ is not birational if and only if, up to translation, $E$
is a $(-1)_A$-invariant sheaf satisfying one of the following
\begin{enumerate}
\item[(i)]{$\chi(A,E) = 1$, i.e., $R\Phi_{\mathcal{P}}(E) \cong \mathcal{O}_{\hat{A}}(-\hat{D})$
where $\hat{D}$ is an ample divisor on $\hat{A}$, and then $deg(\phi) = 2$ when $dim(A) \geq 2$, $deg(\phi) = 2^{rank(E)}$ when $dim(A) = 1$;}
\item[(ii)]{$E \cong \oplus^n\mathcal{O}_A(L)$ where $(A,L)$ is a principal polarization, and then $deg(\phi) = 2$,}
\end{enumerate}
meanwhile there exists an involution $\sigma$ on $P$ such that $\phi$ factors
through the quotient map $P \rightarrow P/(\sigma)$, which fits into the
following commutative diagram
\[\begin{CD}
P      @> >>      P/(\sigma) \\
@V VV               @V VV \\
A       @> >>    A/(-1)_A
\end{CD} \]
\end{Theorem}

\begin{Remark}
Comparing the two properties $IT^0$ and ampleness, they are equivalent for a line bundle on an Abelian variety; and for a vector bundle, Debarre proved that the property $IT^0$ implies ampleness (\cite{De}).
\end{Remark}

\begin{Remark}
The classical proof of Theorem \ref{oldth} is beautiful but very long. Here we not only generalized those classical results but also provided a brief proof thanks to the theory developed by Pareschi and Popa (\cite{PP}, \cite{PP2}). It is key to consider the case when $A$ is simple, and when $A$ is not simple,  with the help of Theorem \ref{lrdcb} and \ref{rdcb}
in Section \ref{pfkey}, we can also understand the morphism $\phi$ well (cf. Sec. \ref{ns}).
\end{Remark}

\section{Definitions and technical results}\label{tool}

In this section, we list some definitions and results which will be used in this paper.

\begin{Theorem}[\cite{Mu}, Thm. 2.2]\label{Mu} Let $A$ be an Abelian variety of dimension $d$. Then
$$R\Psi_{\mathcal{P}} \circ R\Phi_{\mathcal{P}} = (-1)_A^*[-d],~~~R\Phi_{\mathcal{P}} \circ R\Psi_{\mathcal{P}} = (-1)_{\hat{A}}^*[-d]$$
and
$$R\Phi_{\mathcal{P}}\circ (-1)_A^* \cong (-1)_{\hat{A}}^* \circ R\Phi_{\mathcal{P}}, R\Psi_{\mathcal{P}}\circ (-1)_{\hat{A}}^*\cong (-1)_{A}^* \circ R\Psi_{\mathcal{P}}$$
\end{Theorem}

As a corollary we have
\begin{Corollary}\label{Muc}
Let $A$ be an Abelian variety of dimension $d$, $E$ an object on $D^b(A)$, and $F =
R\Phi_{\mathcal{P}}(E)^*$. Then
$$R\Phi_{\mathcal{P}}(E) = F^*, ~~~R\Psi_{\mathcal{P}}(F) = E^*$$
\end{Corollary}
\begin{proof}
We need to show the second equation.
Let $p_1$ and $p_2$ be the projections from $A \times \hat{A}$ to $A$ and $\hat{A}$ respectively. Then we have
\begin{equation}
\begin{split}
R\Psi_{\mathcal{P}}(F) &= R(p_1)_*(Lp_2^*R\mathcal{H}om(R\Phi_{\mathcal{P}}(E), \mathcal{O}_{\hat{A}})\otimes \mathcal{P}) \\
& = R(p_1)_*(R\mathcal{H}om(Lp_2^*R\Phi_{\mathcal{P}}(E), \mathcal{O}_{A \times \hat{A}})\otimes \mathcal{P})\\
& = R(p_1)_*R\mathcal{H}om(Lp_2^*R\Phi_{\mathcal{P}}(E), \mathcal{P})\\
& = R(p_1)_*R\mathcal{H}om(Lp_2^*R\Phi_{\mathcal{P}}(E)\otimes \mathcal{P}^{-1}, \mathcal{O}_{A \times \hat{A}})\\
& = R\mathcal{H}om(R(p_1)_*(Lp_2^*R\Phi_{\mathcal{P}}(E)\otimes \mathcal{P}^{-1}), \mathcal{O}_{A}[-d]) \text{ $\cdot\cdot\cdot$ by Grothendieck duality}\\
& = R\mathcal{H}om(R\Psi_{\mathcal{P}^{-1}}R\Phi_{\mathcal{P}}(E), \mathcal{O}_{A}[-d]) \\
& = R\mathcal{H}om((-1)_{A}^*R\Psi_{\mathcal{P}}R\Phi_{\mathcal{P}}(E), \mathcal{O}_{A}[-d]) \text{ $\cdot\cdot\cdot$ by~$R\Psi_{\mathcal{P}^{-1}} = (-1)_{A}^*R\Psi_{\mathcal{P}}$} \\
& = R\mathcal{H}om(E[-d], \mathcal{O}_{A}[-d]) = E^* \text{ $\cdot\cdot\cdot$ by Theorem \ref{Mu}}
\end{split}
\end{equation}
\end{proof}

\begin{Definition}[\cite{PP2}, Def. 2.1, 2.8, 2.10, \cite{CH2}, Def. 2.6]\label{defgv}
Given a coherent sheaf $\mathcal{F}$ on an Abelian variety $A$, its i-th cohomological support locus is defined as
$$V^i(\mathcal{F}): = \{\alpha \in Pic^0(A)| h^i(\mathcal{F} \otimes \alpha) > 0\}$$
The number $gv(\mathcal{F}): = min_{i>0}\{codim_{Pic^0(A)}V^i(\mathcal{F}) - i\}$ is called the generic vanishing index of $\mathcal{F}$, and
we say $\mathcal{F}$ is a $GV-sheaf$ (resp. $M-regular~ sheaf$) if $gv(\mathcal{F}) \geq 0$ (resp. $>0$).

Let $X$ be an irregular variety equipped with a morphism to an Abelian variety $a: X \rightarrow
A$. Let $\mathcal{F}$ be a sheaf on $X$, its i-th cohomological support locus w.r.t. $a$ is defined as
$$V^i(\mathcal{F}, a) := \{\alpha \in Pic^0(A)| h^i(X, \mathcal{F} \otimes (a^*\alpha)) > 0\}$$
We say $\mathcal{F}$ is full w.r.t. the map $a$ if $V^0(\mathcal{F}, a) = \hat{A}$, and is $continuously~ globally~ generated$ (CGG) w.r.t. $a$ if the sum of the evaluation maps
$$ev_U: \oplus_{\alpha \in U}H^0(\mathcal{F} \otimes \alpha) \otimes (\alpha^{-1}) \rightarrow \mathcal{F}$$
is surjective for any open set $U \subset \hat{A}$.
\end{Definition}

Let's recall the following results due to Pareschi and Popa.
\begin{Proposition}[\cite{PP2}, Prop. 3.1]\label{it0}
Let $\mathcal{F}$ be a GV-sheaf and $H$ a locally free $IT^0$ sheaf on an Abelian variety $A$. Then $\mathcal{F} \otimes H$ satisfies $IT^0$.
\end{Proposition}

\begin{Proposition}[\cite{PP2}, Cor. 5.3]\label{cgg}
An M-regular sheaf on an Abelian variety is CGG.
\end{Proposition}

\begin{Proposition}[\cite{PP}, Prop. 2.12] \label{gg}
Let $X$ be an irregular variety equipped with a morphism to an Abelian variety $a: X \rightarrow
A$. Let $F$ be a coherent sheaf and $L$ a line bundle on $X$. Suppose that both $F$ and $L$ are continuously globally generated w.r.t. $a$. Then $F \otimes L \otimes \alpha$ is globally generated for all $\alpha \in Pic^0(A)$.
\end{Proposition}

\begin{Theorem}[\cite{PP}, Thm. 5.1]\label{ppglb}
Let $F$ be a GV-sheaf on an Abelian variety $A$. Then the following conditions are equivalent:
\begin{enumerate}
\item[(a)]{$F$ is M-regular.}
\item[(b)]{For every locally free $IT^0$ sheaf $H$ on $A$, and for every Zariski open set $U \subset \hat{A}$, the sum of the multiplication of maps of global sections
    $$\oplus_{\alpha \in U}H^0(A, F\otimes \alpha) \otimes H^0(A, H\otimes \alpha^{-1}) \rightarrow H^0(A, F\otimes H)$$
    is surjective.}
\end{enumerate}
\end{Theorem}
\begin{Remark}\label{ppglb}
Modifying the proof of  Thm. 5.1 in \cite{PP}, we can show that for a GV-sheaf $F$ and a locally free $IT^0$ sheaf $H$ on $A$, the sum of the multiplication of maps of global sections
    $$\oplus_{\alpha \in \hat{A}}H^0(A, F\otimes \alpha) \otimes H^0(A, H\otimes \alpha^{-1}) \rightarrow H^0(A, F\otimes H)$$
    is surjective.
\end{Remark}

As a corollary of the theorem above, we get an interesting result.
\begin{Lemma}\label{keylemma}
Let $E$ be an $IT^0$ vector bundle on an Abelian variety $A$. Suppose that $\chi(A, E) = 1$ and that $E$ is a $(-1)_A$-invariant sheaf. Let $\varphi_E: E \rightarrow (-1)_A^*E$ be the corresponding isomorphism. Then naturally $\varphi_E \otimes \varphi_E: E \otimes E \rightarrow (-1)_A^*(E\otimes E)$ shows that $E \otimes E$ is a $(-1)_A$-invariant sheaf, thus $(-1)_A$ induces
an action $(-1)_A^*$ on $H^0(A, E \otimes E)$, moreover if denoting its invariant (resp. anti-invariant) subspace by $H^0(A, E \otimes E)^+$ (resp. $H^0(A, E \otimes E)^-$), then we have $$H^0(A, S^2E) \cong H^0(A, E \otimes E)^+  ~~and~~H^0(A, \wedge^2E) \cong H^0(A, E \otimes E)^- $$
\end{Lemma}
\begin{proof}
For every $\alpha \in \hat{A}$, we can assume $H^0(A, E\otimes \alpha) = span\{e_{\alpha}\}$ since $h^0(A, E\otimes \alpha) = 1$. Applying Theorem \ref{ppglb}, $H^0(A, E \otimes E)$ is spanned by $\{e_{\alpha} \otimes e_{\alpha^{-1}}\}_{\alpha \in \hat{A}}$, and hence $H^0(A, S^2E)$ (resp. $H^0(A, \wedge^2E)$)is spanned by $\{e_{\alpha} \otimes e_{\alpha^{-1}} + e_{\alpha} \otimes e_{\alpha^{-1}} \}_{\alpha \in \hat{A}}$ (resp. $\{e_{\alpha} \otimes e_{\alpha^{-1}} - e_{\alpha} \otimes e_{\alpha^{-1}} \}_{\alpha \in \hat{A}}$). Since $(-1)_A^*e_{\alpha}$ is a section of $(-1)_A^*(E \otimes \alpha) \cong E \otimes \alpha^{-1}$ and $H^0(A, E\otimes \alpha^{-1})$ is spanned by $e_{\alpha^{-1}}$, we can write $(-1)_A^*e_{\alpha} = c e_{\alpha^{-1}}$ with $c \neq 0$, and then $(-1)_A^*e_{\alpha^{-1}} = \frac{1}{c}e_{\alpha}$, hence $(-1)_A^*(e_{\alpha} \otimes e_{\alpha^{-1}})= e_{\alpha^{-1}} \otimes e_{\alpha}$. Finally we find that
$$(-1)_A^*(e_{\alpha} \otimes e_{\alpha^{-1}} + e_{\alpha} \otimes e_{\alpha^{-1}}) = e_{\alpha} \otimes e_{\alpha^{-1}} + e_{\alpha} \otimes e_{\alpha^{-1}}$$ and
$$(-1)_A^*(e_{\alpha} \otimes e_{\alpha^{-1}} - e_{\alpha} \otimes e_{\alpha^{-1}}) = -(e_{\alpha} \otimes e_{\alpha^{-1}} - e_{\alpha} \otimes e_{\alpha^{-1}})$$
So we are done.
\end{proof}

Here we make a simple but very useful remark, which is probably well
known to experts, but we are not able to find a reference.
\begin{Theorem}\label{pf}
Let $X$ and $Y$ be two normal projective varieties, and $\mathcal{L}$ a line bundle on $X \times Y$. Assume $E= (p_2)_*\mathcal{L}$ is a vector bundle and put $P = \mathbb{P}_Y(E)$.
Then there exists an open set $U \subset Y$ such that $\mathbb{P}_U(E)$ parametrizes the divisors in $|\mathcal{L}_y|, y \in U$, correspondingly we get the universal family $\mathcal{D}_U \subset X \times U\rightarrow U$. Denote the closure of $\mathcal{D}_U$ in $X \times P$ by $\mathcal{D}$, which is embedded in $X \times P$ as a divisor. Then we have
$$\mathcal{D} \equiv p^*\mathcal{L} \otimes q^*\mathcal{O}_P(1)$$
where $p,q$ denote the two projections $p: X \times P \rightarrow X
\times Y$, $q: X \times P \rightarrow P$.
\end{Theorem}
\begin{proof}
Let $\{U_\alpha \}_{\alpha \in I}$ be an affine cover of $Y$. Assume $E$ is of rank $n$ and $E(U_\alpha) = \mathcal{O}_Y(U_\alpha)(s_{\alpha}^1, s_{\alpha}^2, ..., s_{\alpha}^n)$ where $s_{\alpha}^i \in H^0(X \times U_\alpha, \mathcal{L}) \cong E(U_\alpha)$, and the divisors $(s_{\alpha}^i)$ have no common vertical components (where a vertical divisor means an effective divisor pulled back via $p_2$); and denote by $x_{\alpha}^i \in E^*(U_\alpha), i=1,2,...,n$ the dual basis of $\{s_{\alpha}^i\}_{i=1,2,...,n}$. Then $x_{\alpha}^1,...x_{\alpha}^n$ can be seen as a coordinate of $E(U_\alpha)$.

The equation $\sum_i x_{\alpha}^i s_{\alpha}^i = 0$ defines a divisor $\mathcal{D}'_\alpha \subset X \times \mathbb{P}_{U_\alpha}(E_{U_\alpha})$,  which coincides with $\mathcal{D}$ if restricted to the open set $X \times \mathbb{P}_{U_\alpha\cap U}(E)$. We patch together the $\mathcal{D}'_\alpha$'s and get a divisor $\mathcal{D}'$. By assumption, the divisor $\mathcal{D}'$ has no vertical part, hence $\mathcal{D}'=\mathcal{D}$. By construction, $\mathcal{D}' \equiv p^*\mathcal{L} \otimes q^*\mathcal{O}_P(1)$, and we are done.
\end{proof}

\section{Proof of Theorem \ref{key}}\label{pfkey}
This section is devoted to prove Theorem \ref{key}.

Let $\hat{P} =
\mathbb{P}_{\hat{A}}(R\Phi_{\mathcal{P}}(E))$, $n = rank(E)$ and $m
= rank(R\Phi_{\mathcal{P}}(E))$. $\pi: P \rightarrow A$ and $\hat{\pi}: \hat{P} \rightarrow \hat{A}$ denote the natural projections, which coincide with their Albanese maps respectively. We denote by $\tilde{\mathcal{P}}$
the pull-back $(\pi \times \hat{\pi})^*\mathcal{P}$ of the Poincar$\acute{e}$ bundle on $A \times \hat{A}$.

By Corollary \ref{Muc}, we conclude that
$$(p_2)_*(p_1^*\mathcal{O}_P(1)\otimes \mathcal{P}_{\pi}) \cong R\Phi_{\mathcal{P}}(\pi_*\mathcal{O}_P(1)) \cong R\Phi_{\mathcal{P}}(E)\cong R^0\Phi_{\mathcal{P}}(E)$$
and
$$(p_1)_*(p_2^*\mathcal{O}_{\hat{P}}(1)\otimes \mathcal{P}_{\hat{\pi}}) \cong R\Psi_{\mathcal{P}}(\hat{\pi}_*\mathcal{O}_{\hat{P}}(1)) \cong R\Psi_{\mathcal{P}}(R\Phi_{\mathcal{P}}(E)^*) \cong E^*$$
where the maps $p_1, p_2$, in the first equation denote the projections from $P \times \hat{A}$ to $P$ and $\hat{A}$ respectively, and in the second equation denote the projections from $A \times \hat{P}$ to $A$ and $\hat{P}$ respectively.
We identify $P$ (resp. $\hat{P}$) with the Hilbert scheme parametrizing the divisors in $\{|\mathcal{O}_{\hat{P}}(1)\otimes \hat{\alpha}||\hat{\alpha} \in Pic^0(\hat{P}) = A\}$ (resp. $\{|\mathcal{O}_{P}(1)\otimes \alpha||\alpha \in Pic^0(P) = \hat{A}\}$), and denote by
$\mathcal{Y} \subset P \times \hat{P}$ the universal family. Using Theorem \ref{pf} we get
\begin{equation}\label{Y}
\mathcal{Y} \equiv p_1^* \mathcal{O}_P(1) \otimes p_2^*\mathcal{O}_{\hat{P}}(1) \otimes \tilde{\mathcal{P}}
\end{equation}

Immediately from Eq. \ref{Y}, it follows that for $x \in P$
$$\mathcal{Y}_x\equiv (p_1^* \mathcal{O}_P(1) \otimes p_2^*\mathcal{O}_{\hat{P}}(1) \otimes \tilde{\mathcal{P}})_{x}\equiv \mathcal{O}_{\hat{P}}(1) \otimes \mathcal{P}_{\pi(x)}$$
Since $R^i\hat{\pi}_*\mathcal{O}_{\hat{P}}(1) = 0$ for $i>0$ and $\hat{\pi}_*\mathcal{O}_{\hat{P}}(1)$ is an $IT^0$ sheaf, we have $h^i(\hat{P}, \mathcal{O}_{\hat{P}}(1)\otimes \mathcal{P}_a) = 0$ for $i>0, a\in A$, and remark that
\begin{itemize}
\item[$\diamondsuit$]{identifying a divisor in $|\mathcal{O}_{P}(1)\otimes \alpha|,\alpha \in \hat{A}$ with a point in $\hat{P}$, for every $x \in P$, $\mathcal{Y}_x \equiv \mathcal{O}_{\hat{P}}(1) \otimes \mathcal{P}_{\pi(x)}$ parametrizes all those divisors passing through $x$; and for every divisor $Y \equiv \mathcal{O}_{\hat{P}}(1) \otimes \mathcal{P}_a, a\in A$ (equivalently, topologically equivalent to $\mathcal{Y}_x$ for some $x \in P$), there exists unique $y \in P$ such that $\mathcal{Y}_y = Y$.}
\end{itemize}
We conclude that for $x, y \in P$,
$$\mathcal{Y}_x \equiv \mathcal{Y}_y \Leftrightarrow \pi(x) = \pi(y), \mathcal{Y}_x = \mathcal{Y}_y \Leftrightarrow x = y$$
And we can write that
\begin{equation}\label{dec}
\mathcal{Y}_x = \mathcal{H}_x + \mathcal{V}_x ~\text{and}~
\mathcal{V}_x= \mathcal{V}^1_x + ... + \mathcal{V}^r_x
\end{equation}
where $\mathcal{H}_x$ is
the horizontal part (if $m = 1$ then $\mathcal{H}_x = \emptyset$), $\mathcal{V}_x = \hat{\pi}^*V_x$ is the vertical part ($\mathcal{V}_x = \emptyset$ if $\mathcal{Y}_x$ is irreducible), and the
$\mathcal{V}^i_x = \hat{\pi}^*V^i_x$'s are the reduced and
irreducible vertical components (two of them may equal). Here we remark that
\begin{itemize}
\item[$\heartsuit$]{as the divisor $V_x$ varies continuously in $\hat{A}$, the divisor $\mathcal{Y}_x = \mathcal{H}_x + \mathcal{V}_x$ varies in the same topological class, correspondingly $x$ varies continuously in $P$ by $\diamondsuit$, so if $x$ is general, we can assume that the divisors $V_x^1, ..., V_x^r, (-1)_{\hat{A}}^*V_x^1,..., (-1)_{\hat{A}}^*V_x^r$ are distinct to each other.}
\end{itemize}

\begin{Remark}\label{lg}
Suppose that for general $x \in P$, $\mathcal{Y}_x$ is reducible. Then there exist an open set $U \subset P$ and two divisors $\mathcal{H}$ and $\mathcal{V}$ on $U \times \hat{P}$, such that for $x \in U$, $\mathcal{H}_x$ (resp. $\mathcal{V}_x$) defined above coincides with the fiber of $\mathcal{H}$ (resp. $\mathcal{V}$) over $x$. We also denote the closure of the two divisors in $P \times \hat{P}$ by $\mathcal{H}$ and $\mathcal{V}$. Then since $\mathcal{Y}_x$ is a divisor in $\hat{P}$ for every $x \in P$, so $\mathcal{Y}$ contains no component which is the pull-back of a divisor in $P$ via the projection $P \times \hat{P} \rightarrow P$, thus $\mathcal{Y}= \mathcal{H} + \mathcal{V}$.
\end{Remark}

\subsection{Base points and the degree of $\phi$.}

\begin{Proposition}\label{base}
$|\mathcal{O}_P(2)|$ is base point free, and $\mathcal{O}_P(3)$ is very ample.
\end{Proposition}
\begin{proof}
Fixing a point $x \in P$. Since $\mathcal{Y}_x$ is a divisor on $P$,
for a general $\alpha \in \hat{A}$, we can find $H_\alpha \in
|\mathcal{O}_P(1)\otimes \alpha|$ and $H_{\alpha^{-1}} \in
|\mathcal{O}_P(1)\otimes \alpha^{-1}|$ such that neither $H_\alpha$
nor $H_{\alpha^{-1}}$ is contained in $\mathcal{Y}_x$, which means $x$
is not contained in $H_\alpha + H_{\alpha^{-1}}$ by $\diamondsuit$. Then we can see that
$|\mathcal{O}_P(2)|$ has no base point.

Let $x \in P$, denote by $I_x$ its ideal sheaf. Consider the following exact sequence
$$0 \rightarrow I_x \otimes \mathcal{O}_P(1) \rightarrow \mathcal{O}_P(1) \rightarrow \mathbb{C}(x) \rightarrow 0$$
Applying $\pi_*$ to the sequence above we obtain an exact sequence on $A$
$$0 \rightarrow \pi_*(I_x \otimes \mathcal{O}_P(1)) \rightarrow E \rightarrow \mathbb{C}(\pi(x)) \rightarrow 0$$
Then applying $R\Phi_{\mathcal{P}}$ to the sequence above, we get
$$0 \rightarrow R^0\Phi_{\mathcal{P}}(\pi_*(I_x \otimes \mathcal{O}_P(1))) \rightarrow R^0\Phi_{\mathcal{P}}(E) \rightarrow \mathcal{P}_{\pi(x)} \rightarrow R^1\Phi_{\mathcal{P}}(\pi_*(I_x \otimes \mathcal{O}_P(1))) \rightarrow 0$$
and that $R^i\Phi_{\mathcal{P}}(\pi_*(I_x \otimes \mathcal{O}_P(1))) = 0$ if $i > 1$ since then $R^i\Phi_{\mathcal{P}}(E) = R^i\Phi_{\mathcal{P}}(\mathbb{C}(\pi(x))) = 0$.

Restricting $\mathcal{Y}_x$ to a general fiber of $\hat{P} \rightarrow A$, it is a hyperplane, so
$$rank(R^0\Phi_{\mathcal{P}}(\pi_*(I_x \otimes \mathcal{O}_P(1)))) = rank(R^0\Phi_{\mathcal{P}}(E)) - 1$$ And since $rank(\mathcal{P}_{\pi(x)})=1$, we conclude that $codim_{\hat{A}}(Supp( R^1\Phi_{\mathcal{P}}(\pi_*(I_x \otimes \mathcal{O}_P(1)))))  \geq 1$,
so $\pi_*(I_x \otimes \mathcal{O}_P(1))$ is a GV-sheaf. Proposition \ref{it0} tells that $\pi_*(I_x \otimes \mathcal{O}_P(1)) \otimes E$ satisfies $IT^0$, thus is CGG by Proposition \ref{cgg}. Since both the natural maps
$$\pi_*(I_x \otimes \mathcal{O}_P(1)) \otimes E \rightarrow \pi_*(I_x \otimes \mathcal{O}_P(2))$$
and
$$\pi^*(\pi_*(I_x \otimes \mathcal{O}_P(2))) \rightarrow I_x \otimes \mathcal{O}_P(2)$$
are surjective, $I_x \otimes \mathcal{O}_P(2)$ is CGG w.r.t. $\pi$. Note that $\mathcal{O}_P(1)$ is CGG w.r.t. $\pi$ since $\pi_*\mathcal{O}_P(1)\cong E$ is CGG and $\pi^*E\rightarrow \mathcal{O}_P(1)$ is surjective. Using Proposition \ref{gg}, it follows that $I_x \otimes \mathcal{O}_P(3)$ is globally generated. Therefore, the line bundle $\mathcal{O}_P(3)$ is very ample.
\end{proof}

\begin{Lemma}\label{spr}
Let $x,y \in P$ be two distinct points. Write that $\mathcal{Y}_x = \mathcal{H}_x +
\mathcal{V}_x $ and $\mathcal{Y}_y = \mathcal{H}_y +
\mathcal{V}_y$ as in \ref{dec}. Then the following conditions are equivalent
\begin{itemize}
\item[(a)]{$|\mathcal{O}_P(2)|$ fails to separate $x,y$;}
\item[(b)]{$\mathcal{H}_x= \mathcal{H}_y$ and $Supp(V_x +(-1)_{\hat{A}}^*V_x) = Supp(V_y +(-1)_{\hat{A}}^*V_y)$.}
\end{itemize}
\end{Lemma}
\begin{proof}
First we show $(a)\Rightarrow (b)$. Now suppose that $|\mathcal{O}_P(2)|$ fails to separate $x,y$.

We claim that $\mathcal{H}_x = \mathcal{H}_y$ if $m >1$. Indeed,
otherwise for a general $\alpha \in \hat{A}$, we can find $H_\alpha
\in |\mathcal{O}_P(1)\otimes \alpha|, H_{\alpha^{-1}} \in
|\mathcal{O}_P(1)\otimes \alpha^{-1}|$ such that $H_\alpha \in \mathcal{Y}_x$ and neither $H_{\alpha}$ nor $H_{\alpha^{-1}}$ is contained in $\mathcal{Y}_y$. So $H_\alpha +
H_{\alpha^{-1}} \in |\mathcal{O}_P(2)|$ contains $x$ but $y$ by $\diamondsuit$, therefore $|\mathcal{O}_P(2)|$ separates
$x,y$.

Let $\alpha \in V_x$. Then $x \in
Bs|\mathcal{O}_P(1)\otimes \alpha|$. Let $\mathcal{V}^i_x = \hat{\pi}^*V^i_x$ be an irreducible component of $\mathcal{Y}_x$. Then for any $\alpha \in V^i_x$, any $H_\alpha \in |\mathcal{O}_P(1) \otimes \alpha| \subset \mathcal{V}^i_x$ and $H_{\alpha^{-1}} \in |\mathcal{O}_P(1) \otimes \alpha^{-1}| \subset \hat{\pi}^*(-1)_{\hat{A}}^*V^i_x$, since $x \in H_\alpha + H_{\alpha^{-1}}$, we have $y \in H_\alpha + H_{\alpha^{-1}}$, so either $H_\alpha \in \mathcal{Y}_y$ or $H_{\alpha^{-1}} \in \mathcal{Y}_y$. Then we conclude that $\mathcal{V}^i_x \subset  \mathcal{Y}_y$ or $\hat{\pi}^*(-1)_{\hat{A}}^*V^i_x \subset  \mathcal{Y}_y$ since $\mathcal{V}^i_x$ is irreducible, so it follows that
$$Supp(V_x) \subset Supp(V_y + (-1)_{\hat{A}}^*V_y)~and~Supp((-1)_{\hat{A}}^*V_x) \subset Supp(V_y + (-1)_{\hat{A}}^*V_y)$$
thus
$$Supp(V_x+ (-1)_{\hat{A}}^*V_x) \subset Supp(V_y + (-1)_{\hat{A}}^*V_y)$$
In the same way, we show that $Supp(V_x+ (-1)_{\hat{A}}^*V_x) \supset Supp(V_y + (-1)_{\hat{A}}^*V_y)$, so one direction follows.

Now we show $(b)\Rightarrow (a)$, so assume $(b)$.

Note that $\pi_*(I_x \otimes \mathcal{O}_P(1))$ is a GV-sheaf which is proved during the proof of Proposition \ref{base}. Again since the natural map
$$\pi_*(I_x \otimes \mathcal{O}_P(1)) \otimes E \rightarrow \pi_*(I_x \otimes \mathcal{O}_P(2))$$
is surjective, using Remark \ref{ppglb}, we conclude that $H^0(A, I_x \otimes \mathcal{O}_P(2))$ is spanned by the elements in the set
$$\{e_\alpha \otimes f_{\alpha^{-1}} \in H^0(P,\mathcal{O}_P(2))|\alpha \in \hat{A}, e_\alpha \in H^0(P, I_x \otimes\mathcal{O}_P(1) \otimes \alpha), f_{\alpha^{-1}} \in H^0(P,\mathcal{O}_P(1) \otimes \alpha^{-1})\}$$

On the other hand reversing the argument when proving $(a)\Rightarrow (b)$, we can prove that
$$x \in H_\alpha + H_{\alpha^{-1}}\Rightarrow y \in H_\alpha + H_{\alpha^{-1}}$$
where $H_\alpha
\in |\mathcal{O}_P(1)\otimes \alpha|, H_{\alpha^{-1}} \in
|\mathcal{O}_P(1)\otimes \alpha^{-1}|$. Then by the fact that $$(e_\alpha \otimes f_{\alpha^{-1}}) \in |\mathcal{O}_P(1) \otimes \alpha| + |\mathcal{O}_P(1) \otimes \alpha^{-1}| \subset |\mathcal{O}_P(2)|$$
we conclude that every divisor in $|\mathcal{O}_P(2)|$ containing $x$ contains $y$. So this direction follows, and we are done.
\end{proof}

\begin{Corollary}\label{dgr}
Let $x$ be a general point, and write that
$\mathcal{Y}_x = \mathcal{H}_x + \mathcal{V}^1_x + ... +
\mathcal{V}^r_x$ as in \ref{dec}. Then the degree of $\phi$ is $2^r$. In
particular, $\phi$ is birational if and only if for a general $x \in P$, $\mathcal{Y}_x$ is irreducible.
\end{Corollary}
\begin{proof}
By $\heartsuit$, we can assume the divisors $V_x^1, ..., V_x^r, (-1)_{\hat{A}}^*V_x^1,..., (-1)_{\hat{A}}^*V_x^r$ are distinct to each other. For a point $y \in P$ distinct to $x$, since $V_x +(-1)_{\hat{A}}^*V_x$ is reduced, using Lemma \ref{spr}, we know that $|\mathcal{O}_P(2)|$ fails to separate $x$ and $y$ if and only if
$$V_x +(-1)_{\hat{A}}^*V_x = V_y +(-1)_{\hat{A}}^*V_y$$
equivalently
$$V_y =  ((-1)_{\hat{A}}^{\epsilon_1})^*V^1_x + ... +
((-1)_{\hat{A}}^{\epsilon_r})^*V^r_x, ~\epsilon_i \in \{0,1\}, i = 1,2,...,r$$

On the other hand, for every choice $\epsilon_i \in \{0,1\}, i = 1,2,...,r$, since $((-1)_{\hat{A}}^{\epsilon_1})^*V^i_x - V^i_x \in Pic^0(\hat{A})$, there exists $\alpha \in Pic^0(\hat{P})$ such that
$$\mathcal{H}_x + \hat{\pi}^*((-1)_{\hat{A}}^{\epsilon_1})^*V^1_x + ... +
\hat{\pi}^*((-1)_{\hat{A}}^{\epsilon_r})^*V^r_x \equiv \mathcal{Y}_x + \alpha$$
thus there exists unique $y \in P$ such that
$$\mathcal{Y}_y = \mathcal{H}_x + \hat{\pi}^*((-1)_{\hat{A}}^{\epsilon_1})^*V^1_x + ... +
\hat{\pi}^*((-1)_{\hat{A}}^{\epsilon_r})^*V^r_x$$

Then we conclude that $deg(\phi) = 2^r$, and the remaining assertion follows easily.
\end{proof}

\subsection{The reducibility of the general hyperplane
section}\label{rdcbl}

\begin{Theorem}\label{rdcb}
For a general $x \in P$, if $\mathcal{Y}_x$ has non-trivial vertical part, then $A \cong A_1 \times A_2$ ($A_1,A_2$ maybe a point), and there exists a line
bundle $L_1$ on $A_1$ and a vector bundle $E_2$ on $A_2$ such that
$E \cong L_1 \boxtimes E_2$ and that $\chi(A_1, L_1) = 1$ or
$\chi(A_2, E_2) = 1$.

The converse is also true.
\end{Theorem}
\begin{proof}
If $m = \chi(E)= 1$, then $R^0\Phi_{\mathcal{P}}(E)$ is a line bundle on $\hat{A}$. Let $A_1=pt$, $A_2=A$, $L_1$ the trivial line bundle on $A_1$ and $E_2=E$ on $A_2$. Then we are done.

Now assume $m > 1$.
By assumption for a general $x$, we can write
$\mathcal{Y}_x = \mathcal{H}_x + \mathcal{V}_x$ as in \ref{dec}, correspondingly write that $\mathcal{Y} = \mathcal{H} + \mathcal{V}$ as in Remark \ref{lg}. In the following, for every $x \in P$, the divisors $\mathcal{H}_x$ and $\mathcal{V}_x$ denote the fibers of $\mathcal{H}$ and $\mathcal{V}$ over $x$ respectively.

Let $a \in A$ be general. First we claim that for two general points $x_1, x_2 \in P$ over $a$, $\mathcal{H}_{x_1} \equiv \mathcal{H}_{x_2}$.
Indeed, for a general point $x_0$ fixed over $a$, we can define a morphism
$\pi^{-1}a \rightarrow A$ via $x \mapsto \mathcal{H}_x -
\mathcal{H}_{x_0}$ by identifying $\mathcal{H}_x - \mathcal{H}_{x_0}$ with
an element in $Pic^0(\hat{A}) \cong A$. This
map must be constant since $\pi^{-1}a$ is rational. So this
assertion is true. In turn we conclude that
\begin{equation}\label{syst}
|\mathcal{O}_{\hat{P}}(1) \otimes a| = |\mathcal{H}_x + \mathcal{V}_x| = |\mathcal{H}_x| + |\mathcal{V}_x| \text{~where~$x\in P$~is~a~general~point~over~$a$}
\end{equation}
and that $|\mathcal{H}_x| = \mathcal{H}_x$ or $|\mathcal{V}_x| =
\mathcal{V}_x$ holds true due to Bertini's theorem. By semicontinuity, we conclude that both $h^0(\hat{P}, \mathcal{H}_x)$ and $h^0(\hat{P}, \mathcal{V}_x)$ are invariant as $x$ varies in $P$, and one of them is equal to $m$ and the other is equal to $1$.

Therefore, for a general $a_0 \in A$ fixed, we can define two rational maps $\iota_1,\iota_2:A \dashrightarrow A$
by $\iota_1: a \mapsto \mathcal{H}_x - \mathcal{H}_{x_0}$ and $\iota_2:
a \mapsto \mathcal{V}_x - \mathcal{V}_{x_0}$ where $x,x_0$ are two general points over $a,a_0$ respectively, which are well defined and extend to two morphisms. Denote their images by
$A_1$ and $A_2$, which are two sub-torus passing through $0 \in A$ hence subgroups of $A$.

Let $a \in A_1 \cap A_2$, then $a^{-1} \in A_1 \cap A_2$. We can find $x_1,x_2 \in P$ such that $\iota_1(\pi(x_1))=a,\iota_2(\pi(x_2))=a^{-1}$. It follows that $|\mathcal{H}_{x_1}| = |\mathcal{H}_{x_0} \otimes a| \neq \emptyset$ and $|\mathcal{V}_{x_2}|=|\mathcal{V}_{x_0} \otimes a^{-1}| \neq \emptyset$,
and therefore,
$$|\mathcal{H}_{x_0} \otimes a| + |\mathcal{V}_{x_0} \otimes (a^{-1})| = |\mathcal{H}_{x_1}| + |\mathcal{V}_{x_2}| \subset |\mathcal{O}_{\hat{P}}(1) \otimes a_0|$$
Note that a general divisor in $|\mathcal{H}_{x_0}|$ is irreducible since $a_0$ is general, $h^0(\hat{P}, \mathcal{H}_{x_0})= h^0(\hat{P}, \mathcal{H}_{x_1})$, and $\mathcal{H}_{x_0} \sim_{num} \mathcal{H}_{x_1}$.
So Eq. \ref{syst} implies that $\mathcal{H}_{x_0} \equiv \mathcal{H}_{x_1}$, thus $a = 0$. Therefore, $A_1 \cap A_2 = \{0\}$.

On the
other hand, for general $a \in A$ and $x \in P$ over $a$, by definition it follows that
$$|\mathcal{H}_x| = |\mathcal{H}_{x_0} \otimes \iota_1(a)| \neq \emptyset, |\mathcal{V}_x| = |\mathcal{V}_{x_0} \otimes \iota_2(a)| \neq \emptyset$$
hence $|\mathcal{O}_{\hat{P}}(1) \otimes a| =
|\mathcal{H}_{x_0} \otimes \iota_1(a)| + |\mathcal{V}_{x_0} \otimes \iota_2(a)|$. By $|\mathcal{O}_{\hat{P}}(1) \otimes a_0| =
|\mathcal{H}_{x_0}| + |\mathcal{V}_{x_0}|$, we conclude that $a - a_0
= \iota_1(a) + \iota_2(a)$, consequently $A \cong A_1 \times A_2$.

Let $\hat{A_i} = Pic^0(A_i)$ and denote by $\mathcal{P}_i$ the pull-back
of the Poincar$\acute{e}$ bundle on $A_i \times \hat{A_i}$ via the map $A_i \times \hat{A} \rightarrow A_i \times
\hat{A}_i$. By the analysis above we conclude that
$$R\Psi_{\mathcal{P}}(\hat{\pi}_*\mathcal{O}_{\hat{P}}(1) \otimes a_0) \cong
R^0\Psi_{\mathcal{P}_1}(\hat{\pi}_*\mathcal{O}_{\hat{P}}(\mathcal{H}_{x_0}))
\boxtimes
R^0\Psi_{\mathcal{P}_2}(\hat{\pi}_*\mathcal{O}_{\hat{P}}(\mathcal{V}_{x_0}))$$

Then recalling that $|\mathcal{H}_x| = \mathcal{H}_x$ or
$|\mathcal{V}_x| = \mathcal{V}_x$ holds, we conclude that at least one of
$R^0\Psi_{\mathcal{P}_1}(\hat{\pi}_*\mathcal{O}_{\hat{P}}(\mathcal{H}_{x_0}))$ and
$R^0\Psi_{\mathcal{P}_2}(\hat{\pi}_*\mathcal{O}_{\hat{P}}(\mathcal{V}_{x_0}))$ is a
line bundle. Then since $R\Psi_{\mathcal{P}}(\hat{\pi}_*\mathcal{O}_{\hat{P}}(1) \otimes a_0) \cong t_{-a_0}^*E^*$ where $t_{-a_0}$ is the translation by $-a_0$, so one direction is completed.

The converse assertion is obvious, so this theorem is true.
\end{proof}

\subsection{The map $\phi$}\label{pfmain} We distinguish the two cases $m = 1$ and $m > 1$.

\textbf{Case $m = 1$:} In this case, for every $x \in P$, $\mathcal{Y}_x$ is a divisor $V_x$ on $\hat{A}$, hence $\phi$ is surely not birational by Corollary \ref{dgr}, and $E = R\Psi_{\mathcal{P}}(\mathcal{O}_{\hat{A}}(V))^*$ for some ample divisor $V$ on $\hat{A}$.
Find a suitable $a_0 \in A$ such that $\mathcal{O}_{\hat{A}}(V)\otimes a_0$ is symmetric. Then its Fourier-Mukai transform $R\Psi_{\mathcal{P}}(\mathcal{O}_{\hat{A}}(V)\otimes a_0) \cong t_{-a_0}^*E^*$ is a $(-1)_A$-invariant sheaf. Replacing $E$ by $t_{-a_0}^*E$, we can assume $E$ is a $(-1)_A$-invariant sheaf, so $(-1)_A$ induces an action $(-1)_A^*$ on $H^0(A, E \otimes E)$.
Note that $H^0(P, \mathcal{O}_P(2)) \cong H^0(A, S^2E)$. By Lemma \ref{keylemma}, $(-1)_A^*$ is identity on $H^0(A, S^2E)$,
hence $\phi$ factors through an involution $\sigma$ making the following commutative diagram hold
\[\begin{CD}
P      @> >>       P/(\sigma) \\
@V\pi VV               @V\pi' VV \\
A       @> >>    A/((-1)_A)
\end{CD} \]

In particular, if $A$ is simple, then Corollary \ref{dgr} gives that $deg(\phi) = 2$ if $dim(A) > 1$ since then a general $V_x$ is irreducible, and $deg(\phi) = 2^n$ if $dim(A) = 1$ since then $deg(V_x) = rank(E) = n$.

\textbf{Case $m > 1$:} In this case, using Corollary \ref{dgr}, immediately it follows that $\phi$ is not birational if and only if for a general $x \in P$, $\mathcal{Y}_x$ is
reducible.

If $A$ is simple and $\phi$ is not birational, applying Theorem
\ref{rdcb}, we conclude that $A \cong A_1 \times A_2$ and hence either $A_1$ or $A_2$ is a point. So we can assume $A_1 = A$, $L_1$ be a line bundle on $A_1$ such that $(A,L_1)$ is a principally
polarization, $A_2=pt$ and $E_2$ is a vector bundle of rank $>1$. Therefore $P \cong A \times \mathbb{P}^{n-1}$, then
reducing to the case when $m = 1$, we can show $\phi$ is of degree 2.

In conclusion, we finally proved Theorem \ref{key}.

\subsection{The degree of $\phi$ when $A$ is not simple}\label{ns}
When $A$ is not simple, to study the degree of $\phi$, it suffices to consider the case when $m=1$, i.e., $R\Phi_{\mathcal{P}}(E)$ is quasi-isomorphic to a line bundle.
This is completely clear thanks to the following classical result.
\begin{Theorem}[\cite{LB} Sec. 3.4]\label{lrdcb}
Let $A$ be an Abelian variety, $D$ an ample divisor on $A$. Then the linear
system $|D|$ can be written as $|D| = |D_1| + D_2 + ...+ D_r$ where
\begin{enumerate}
\item[(i)]{$|D_1|$ is the moving part;}
\item[(ii)]{$A \cong A_1 \times A_2 \times ... \times A_r$ where $A_i, i =2,3,...,r$ is simple;}
\item[(iii)]{for $i = 1,2,...,r$, $D_i \equiv p_i^*B_i$ where $B_i$ is a divisor on $A_i$, and for $i=2,...,r$,
$(A_i, B_i)$ is a principal polarization.}
\end{enumerate}
Moreover if $dim(A_1) > 1$, then a general element in $|B_1|$ is irreducible.
\end{Theorem}

\begin{Corollary}
Assume that $R\Phi_{\mathcal{P}}(E) = \mathcal{O}_{\hat{A}}(-\hat{D})$ for some ample divisor $\hat{D}$ on $\hat{A}$, and the linear
system $|\hat{D}| = |\hat{D}_1| +\hat{D}_2 + ...+ \hat{D}_r$ as in Theorem \ref{lrdcb}, where $\hat{A} \cong \hat{A}_1 \times \hat{A}_2 \times ... \times \hat{A}_r$ and $\hat{D}_i = p_i^*\hat{B_i}$. Then we have
\begin{itemize}
\item[(1)]{if $\hat{A}_1$ is an elliptic curve, then $deg(\phi) = 2^{deg(\hat{B}_1) + r-1}$;}
\item[(2)]{if $dim(\hat{A}_1) > 1$, then $deg(\phi) = 2^{r}$}
\end{itemize}
\end{Corollary}
\begin{proof}
For $a \in A$, there exists $\alpha \in \hat{A}$ such that $\mathcal{O}_{\hat{A}}(\hat{D})\otimes \mathcal{P}_a \equiv t_{\alpha}^*\hat{D}$, hence
$$|\mathcal{O}_{\hat{A}}(\hat{D})\otimes \mathcal{P}_a| = |t_{\alpha}^*\hat{D}_1| + t_{\alpha}^*\hat{D}_2 + ...+ t_{\alpha}^*\hat{D}_r$$
For $x \in P$, the divisor $\mathcal{V}_x \in |\mathcal{O}_{\hat{A}}(\hat{D})\otimes \mathcal{P}_{\pi(x)}|$. Then our assertions follow from Corollary \ref{dgr}.
\end{proof}

\subsection{Applications}
The results and the ideas of the proof probably find their applications when considering the morphism defined by the square of a line bundle on an irregular variety. Here we prove that
\begin{Proposition}\label{big}
Let $a: X \rightarrow A$ be a generically finite morphism to an Abelian variety and $V$ a Cartier divisor on $X$ such that $\mathcal{O}_X(V)$ is full w.r.t. $a$ (see Def. \ref{defgv}). Then $\mathcal{O}_X(V)$ is CGG at general points of $X$ w.r.t. $a$, and the linear system $|2V|$ defines a generically finite map. In particular, the divisor $V$ is big.
\end{Proposition}
\begin{proof}
If $a^*: Pic^0(A) \rightarrow Pic^0(X)$ is not an embedding, then we get a factorization $$a = \iota \circ a': X \rightarrow A' \rightarrow A,$$ where $A'$ is the Abelian variety with dual $Pic^0(A') = a^*Pic^0(A) \subset Pic^0(X)$, and $\iota: A' \rightarrow A$ arises from the dual map $a^*: Pic^0(A) \rightarrow a^*Pic^0(A) = Pic^0(A')$. Note that $\iota^*: Pic^0(A) \rightarrow Pic^0(A')$ is finite and surjective, hence maps an open set to an open set. Since $H^0(X, F \otimes a^*\alpha) = H^0(X, F \otimes a'^*\iota^*\alpha)$ for a sheaf $F$ on $X$ and $\alpha \in \hat{A}$, we conclude that for a sheaf $F$ on $X$,
$$F ~is ~CGG~ w.r.t. ~a \Leftrightarrow F ~is ~CGG~ w.r.t. ~a'$$ and $$F ~is ~full~ w.r.t. ~a \Leftrightarrow F ~is ~full~ w.r.t. ~a'.$$ Then we consider the map $a'$ instead, so we can assume $a^*: Pic^0(A) \rightarrow Pic^0(X)$ is injective in the following.

Consider the line bundle $\mathcal{L}: = p_1^*\mathcal{O}_X(V) \otimes \mathcal{P}_a$ on $X \times \hat{A}$. Put
$$\hat{P}= Proj_{\mathcal{O}_A}(\oplus S^n ((p_2)_*\mathcal{L})^*)$$
Then there exists an open set $U$ of $\hat{P}$ parametrizes the divisors in $|\mathcal{O}_X(V)\otimes \alpha|, \alpha \in \hat{A}$. We can find an open set $U_0 \subset \hat{A}$ and a section $s: U_0 \rightarrow U$. Identifying $U_0$ with $s(U_0)$, we get a universal divisor
$$\mathcal{Y} \subset X \times U_0$$
and denote its closure in $X \times \hat{A}$ by $\mathcal{V}$. Then $\mathcal{V}_\alpha \equiv\mathcal{O}_X(V) \otimes \alpha$ for $\alpha \in U_0$,

As in \cite{BLNP} Sec. 5, take a general $\alpha_0 \in U_0$, and define a map
$$f_{\alpha_0}: Pic^0(A) \rightarrow Pic^0(X),  \alpha \mapsto \mathcal{O}_X(\mathcal{V}_\alpha - \mathcal{V}_{\alpha_0})$$
which extends to a morphism. By rigidity,
$$f: = f_{\alpha_0} - f_{\alpha_0}(0):  Pic^0(A) \rightarrow Pic^0(X)$$ is a homomorphism.

By definition, the image of $f$ coincides with $a^*Pic^0(A)$. Identifying $\hat{A} = Pic^0(A)$ with $a^*Pic^0(A)$, we have $f = id_{\hat{A}}$. Then arguing as in Sec. 5.3 in \cite{BLNP}, we prove
$$\mathcal{V} \equiv p_1^*\mathcal{O}_X(V) \otimes \mathcal{P}_a \otimes p_2^*\mathcal{O}_{\hat{A}}(\mathcal{V}_p)$$
where $p\in X$ is a point mapped to $0 \in A$ via $a$.
Note that for $x\in X$, $\mathcal{V}_x = \{\alpha \in \hat{A}|x\in \mathcal{V}_\alpha\}$, and for two distinct points $x, y \in X$, if $a(x) \neq a(y)$, then $\mathcal{V}_x \neq \mathcal{V}_y$ since they are not even linearly equivalent. Since the map $a$ is finite, considering the universal divisor $\mathcal{V}$ and arguing as in the proof of Lemma \ref{spr}, we can show $|2V|$ defines a generically finite map.

Identifying $\hat{A}$ with a family of divisors on $X$ via $\alpha \mapsto \mathcal{V}_\alpha$, then for a general point $x \in X$, $\mathcal{V}_x$ is a divisor on $\hat{A}$, which parametrizes the divisors passing through $x$. So it is easy to conclude that $\mathcal{O}_X(V)$ is CGG at general points of $X$.
\end{proof}

\end{document}